\documentclass[11pt,leqno]{article}
\usepackage[margin=1in]{geometry} 
\usepackage{amssymb,amsfonts,amsmath,bbm,mathrsfs,stmaryrd,mathtools}
\usepackage{xcolor}
\usepackage{url}

\usepackage{extarrows}

\usepackage[shortlabels]{enumitem}
\usepackage{tensor}

\usepackage{xr}
\usepackage[T1]{fontenc}

\usepackage[colorlinks,
linkcolor=black!75!red,
citecolor=blue,
pdftitle={},
pdfproducer={pdfLaTeX},
pdfpagemode=None,
bookmarksopen=true,
bookmarksnumbered=true,
backref=page]{hyperref}

\usepackage{tikz}
\usetikzlibrary{arrows,calc,decorations.pathreplacing,decorations.markings,intersections,shapes.geometric,through,fit,shapes.symbols,positioning,decorations.pathmorphing}

\usepackage{braket}

\usepackage[amsmath,thmmarks,hyperref]{ntheorem}
\usepackage{cleveref}

\newcommand\nthalias[1]{\AddToHook{env/#1/begin}{\crefalias{lemma}{#1}}}

\nthalias{definition}
\nthalias{example}
\nthalias{examples}
\nthalias{remark}
\nthalias{remarks}
\nthalias{convention}
\nthalias{notation}
\nthalias{construction}
\nthalias{theoremN}
\nthalias{propositionN}
\nthalias{corollaryN}
\nthalias{lemma}
\nthalias{proposition}
\nthalias{corollary}
\nthalias{theorem}
\nthalias{conjecture}
\nthalias{question}
\nthalias{assumption}

\creflabelformat{enumi}{#2#1#3}

\crefname{section}{Section}{Sections}
\crefformat{section}{#2Section~#1#3} 
\Crefformat{section}{#2Section~#1#3} 

\crefname{subsection}{\S}{\S\S}
\AtBeginDocument{%
  \crefformat{subsection}{#2\S#1#3}%
  \Crefformat{subsection}{#2\S#1#3}%
}

\crefname{subsubsection}{\S}{\S\S}
\AtBeginDocument{%
  \crefformat{subsubsection}{#2\S#1#3}%
  \Crefformat{subsubsection}{#2\S#1#3}%
}

\theoremstyle{plain}

\newtheorem{lemma}{Lemma}[section]

\theoremstyle{plain}
\theoremnumbering{Alph}
\newtheorem{theoremN}{Theorem}

\newtheorem{corollaryN}[theoremN]{Corollary}

\theoremstyle{plain}
\theorembodyfont{\upshape}
\theoremsymbol{\ensuremath{\blacklozenge}}

\crefname{definition}{definition}{definitions}
\crefformat{definition}{#2definition~#1#3} 
\Crefformat{definition}{#2Definition~#1#3} 

\crefname{ex}{example}{examples}
\crefformat{example}{#2example~#1#3} 
\Crefformat{example}{#2Example~#1#3} 

\crefname{exs}{example}{examples}
\crefformat{examples}{#2example~#1#3} 
\Crefformat{examples}{#2Example~#1#3} 

\crefname{remark}{remark}{remarks}
\crefformat{remark}{#2remark~#1#3} 
\Crefformat{remark}{#2Remark~#1#3} 

\crefname{remarks}{remark}{remarks}
\crefformat{remarks}{#2remark~#1#3} 
\Crefformat{remarks}{#2Remark~#1#3} 

\crefname{convention}{convention}{conventions}
\crefformat{convention}{#2convention~#1#3} 
\Crefformat{convention}{#2Convention~#1#3} 

\crefname{notation}{notation}{notations}
\crefformat{notation}{#2notation~#1#3} 
\Crefformat{notation}{#2Notation~#1#3} 

\crefname{table}{table}{tables}
\crefformat{table}{#2table~#1#3} 
\Crefformat{table}{#2Table~#1#3}

\crefname{lemma}{lemma}{lemmas}
\crefformat{lemma}{#2lemma~#1#3} 
\Crefformat{lemma}{#2Lemma~#1#3} 

\crefname{proposition}{proposition}{propositions}
\crefformat{proposition}{#2proposition~#1#3} 
\Crefformat{proposition}{#2Proposition~#1#3} 

\crefname{propositionN}{proposition}{propositions}
\crefformat{propositionN}{#2proposition~#1#3} 
\Crefformat{propositionN}{#2Proposition~#1#3} 

\crefname{corollary}{corollary}{corollaries}
\crefformat{corollary}{#2corollary~#1#3} 
\Crefformat{corollary}{#2Corollary~#1#3} 

\crefname{corollaryN}{corollary}{corollaries}
\crefformat{corollaryN}{#2corollary~#1#3} 
\Crefformat{corollaryN}{#2Corollary~#1#3} 

\crefname{theorem}{theorem}{theorems}
\crefformat{theorem}{#2theorem~#1#3} 
\Crefformat{theorem}{#2Theorem~#1#3} 

\crefname{theoremN}{theorem}{theorems}
\crefformat{theoremN}{#2theorem~#1#3} 
\Crefformat{theoremN}{#2Theorem~#1#3} 

\crefname{enumi}{}{}
\crefformat{enumi}{#2#1#3}
\Crefformat{enumi}{#2#1#3}

\crefname{assumption}{assumption}{Assumptions}
\crefformat{assumption}{#2assumption~#1#3} 
\Crefformat{assumption}{#2Assumption~#1#3} 

\crefname{construction}{construction}{Constructions}
\crefformat{construction}{#2construction~#1#3} 
\Crefformat{construction}{#2Construction~#1#3} 

\crefname{question}{question}{Questions}
\crefformat{question}{#2question~#1#3} 
\Crefformat{question}{#2Question~#1#3} 

\crefname{equation}{}{}
\crefformat{equation}{(#2#1#3)} 
\Crefformat{equation}{(#2#1#3)}

\numberwithin{equation}{section}

\theoremstyle{nonumberplain}
\theoremsymbol{\ensuremath{\blacksquare}}

\newtheorem{proof}{Proof}
\newcommand\pf[1]{\newtheorem{#1}{Proof of \Cref{#1}}}

\newcommand\bG{{\mathbb G}}

\newcommand\bK{{\mathbb K}}
\newcommand\bL{{\mathbb L}}
\newcommand\bM{{\mathbb M}}

\newcommand\bR{{\mathbb R}}

\newcommand\bT{{\mathbb T}}

\newcommand\bZ{{\mathbb Z}}

\newcommand\cE{{\mathcal E}}
\newcommand\cF{{\mathcal F}}

\newcommand\fg{{\mathfrak g}}
\newcommand\fh{{\mathfrak h}}

\newcommand\fk{{\mathfrak k}}
\newcommand\fl{{\mathfrak l}}

\newcommand\ft{{\mathfrak t}}

\DeclareMathOperator{\Ad}{Ad}

\DeclareMathOperator{\im}{\mathrm{im}}

\newcommand{\qedhere}{\mbox{}\hfill\ensuremath{\blacksquare}}

\newcommand{\xrightarrowdbl}[2][]{%
  \xrightarrow[#1]{#2}\mathrel{\mkern-14mu}\rightarrow
}

\title{Pointwise-relatively-compact subgroups and trivial-weight-free representations}
\author{Alexandru Chirvasitu}

\begin{document}

\date{}

\newcommand{\Addresses}{{
  \bigskip
  \footnotesize

  \textsc{Department of Mathematics, University at Buffalo}
  \par\nopagebreak
  \textsc{Buffalo, NY 14260-2900, USA}  
  \par\nopagebreak
  \textit{E-mail address}: \texttt{achirvas@buffalo.edu}


}}

\maketitle

\begin{abstract}
  A pointwise-elliptic subset of a topological group is one whose elements all generate relatively-compact subgroups. A connected locally compact group has a dense pointwise-elliptic subgroup if and only if it is an extension by a compact normal subgroup of a semidirect product $\mathbb{L}\rtimes \mathbb{K}$ with connected, simply-connected Lie $\mathbb{L}$, compact Lie $\mathbb{K}$, with the commutator subgroup $\mathbb{K}'$ acting on the Lie algebra $Lie(\mathbb{L})$ with no trivial weights. This extends and recovers a result of Kabenyuk's, providing the analogous classification with $\mathbb{G}$ assumed Lie connected, topologically perfect, with no non-trivial central elliptic elements.
\end{abstract}


\vspace{.5cm}

\noindent{MSC 2020: 22E15; 54D30; 22D05; 22D12; 22E25; 22E60; 54E52; 17B05
  


}


\section*{Introduction}

A discussion in \cite[paragraph preceding Proposition 2.5]{zbMATH03654393} asks whether connected Lie groups $\bG$ with dense subgroups $\Gamma\le \bG$ consisting of elements which each generate relatively-compact subgroups are necessarily compact. Motivated by this, \cite[Theorem 2]{zbMATH03975248} (also \cite[Theorem 1.6]{MR3578406}) classifies those connected Lie groups
\begin{itemize}[wide]
\item which have dense subgroups $\Gamma\le \bG$ as above;
\item are \emph{topologically perfect} in the sense of having dense derived subgroups;
\item and have no non-trivial central compact subgroups:
\end{itemize}
they are precisely the semidirect products $\bL\rtimes \bK$ with connected, simply-connected, nilpotent $\bL$ and compact, semisimple $\bK$ so that the center $Z(\bK)$ acts on $\bL$ with no non-trivial fixed point. This note extends that characterization (providing a few other equivalent ones) to arbitrary connected locally compact groups. 

For subsets $A,B\subseteq \bG$ of a group we write
\begin{equation*}
  [A,B]:=\left\{[a,b]:=aba^{-1}b^{-1}\ |\ a\in A,\ b\in B\right\}
\end{equation*}
for the set of \emph{commutators} of elements in $A$ and $B$ respectively.

We will refer freely to the notion of a \emph{(pro-)Lie algebra} $Lie(\bG)$ \cite[Definitions 2.6 and 3.10 and Theorem 3.12]{hm_pro-lie-bk} of a connected locally compact group $\bG$, for which theory we refer the reader to \cite{hm_pro-lie-bk} with more specific references where appropriate. The notation will be consistent: $\fg:=Lie(\bG)$, $\fh:=Lie(\bK)$, $\fl:=Lie(\bL)$, etc. A few more vocabulary items will help unwind the result paraphrased in abbreviated form in the abstract.

\begin{itemize}[wide]
\item A \emph{weight} of a compact group $\bK$ with respect to a complex finite-dimensional representation $\bK\xrightarrow{\rho} \mathrm{GL}(V)$ is a 1-dimensional sub-representation of the restriction $\rho|_{\bT}$ for a maximal \emph{pro-torus} (i.e. \cite[Definitions 9.30]{hm5} compact connected abelian subgroup) $\bT\le \bK_0$ of the identity component of $\bK$.

  The term applies to real representations as well, after complexification. 
  
\item Subsets of (always Hausdorff) topological groups are \emph{elliptic} when contained in compact subgroups, and \emph{pointwise-elliptic} when they consist of elliptic elements. 
\end{itemize}

\begin{theoremN}\label{thn:dns.ell.sbgp}
  The following conditions on a locally compact connected group $\bG$ are equivalent.
  
  \begin{enumerate}[(1),wide=0pt]
  \item\label{item:thn:dns.ell.sbgp:pt.ell.sbg} $\bG$ has a dense pointwise-elliptic subgroup.

  \item\label{item:thn:dns.ell.sbgp:pt.ell.sbs} $\bG$ has a dense pointwise-elliptic subset $\Gamma$ with $[\Gamma,\Gamma]$ pointwise-elliptic.

  \item\label{item:thn:dns.ell.sbgp:char} The conjunction of
    
    \begin{enumerate}[(a),wide]
    \item\label{item:thn:dns.ell.sbgp:cpct.ss} the quotient $\bG/\bG_{sol}$ by the \emph{(solvable) radical} \cite[Definition 10.23]{hm_pro-lie-bk} $\bG_{sol}\le \bG$ is compact, and

    \item\label{item:thn:dns.ell.sbgp:0wt} for every (equivalently, any one) maximal compact subgroup $\bK\le \bG$ the adjoint action of the derived subgroup $\bK'\le \bK$ on the quotient space $\fg/\fk$ has no trivial weights.     
    \end{enumerate}
  \item\label{item:thn:dns.ell.sbgp:alm.semidir} $\bG$ is expressible as an extension
    \begin{equation}\label{eq:lk.by.m}
      \{1\}\to
      \left(\bM\text{ compact}\right)
      \lhook\joinrel\xrightarrow{\quad}
      \bG
      \xrightarrowdbl{\quad}
      \bL\rtimes \bK
      \to \{1\}
    \end{equation}
    with $\bL$ Lie, connected, simply-connected and nilpotent, and $\bK'$ acting on $\fl:=Lie(\bL)$ trivial-weight-freely.

    Equivalently, maximal tori $\bT\le \bK'$ act on $\bL$ with no non-trivial fixed points. 
  \end{enumerate}
\end{theoremN}

The fact that ``many'' free tuples generating pointwise-elliptic groups exist when $\bG$ is Lie will be a byproduct. Recall \cite[pre Theorem 9.2]{oxt_meas-cat_2e_1980} that a \emph{residual} subset of a locally compact Hausdorff topological space is one which contains a countable intersection of dense opens. 

\begin{corollaryN}\label{corn:when.lie}
  Let $\bG$ be a non-solvable connected Lie group satisfying the equivalent conditions of \Cref{thn:dns.ell.sbgp}.

  For sufficiently large positive integers $d$ there are $d$-tuples in $\bG$ generating dense free pointwise-elliptic subgroups of $\bG$, and the set of such tuples is residual.  \qedhere
\end{corollaryN}


\subsection*{Acknowledgments}

I am grateful for comments, suggestions and literature pointers from E. Breuillard and T. Gelander. 


\section{Dense subgroups consisting of elliptic elements}\label{se:lg.pt.ell}

We follow \cite[Definitions 10.1 and 10.8]{hm_pro-lie-bk} in denoting by
\begin{equation}\label{eq:der.ser}
  \bG=:\bG^{(0)}\ge \bG^{(1)}\ge \cdots
  \quad\text{and}\quad
  \bG=:\bG^{((0))}\ge \bG^{((1))}\ge \cdots
\end{equation}
the \emph{derives} and \emph{closed derived series} of a topological group respectively. 

\pf{thn:dns.ell.sbgp}
\begin{thn:dns.ell.sbgp}
  The mutual conjugacy \cite[\S 4.13, Theorem, pp.188-189]{mz} of the maximal compact subgroups shows that indeed the quantifier in condition \Cref{item:thn:dns.ell.sbgp:0wt} does not matter. Moreover, expressing $\bG$ as
  \begin{equation*}
    \bG
    \cong
    \varprojlim_{\substack{\bL\trianglelefteq \bG\\\bL\text{ compact}\\\bG/\bL\text{ Lie}}}
    \bG/\bL
    \quad
    \left(\text{\cite[\S 4.6, Theorem]{mz}}\right),
  \end{equation*}
  we can assume $\bG$ lie upon observing that an element is elliptic if and only if its image through a quotient by a compact normal subgroup is, so that the pointwise ellipticity of a subset $\Gamma\le \bG$ is equivalent to that of
  \begin{equation*}
    \bL\cdot\Gamma\le \bG
    \quad\text{and hence}\quad
    \pi(\Gamma)\cong \bL\cdot\Gamma/\bL
    ,\quad
    \bG\xrightarrow{\quad\pi\quad}\bG/\bL
  \end{equation*}
  for any (equivalently, all) compact normal $\bL\trianglelefteq \bG$.

  As one last preliminary simplification, note that \Cref{item:thn:dns.ell.sbgp:cpct.ss} and the weaker version of \Cref{item:thn:dns.ell.sbgp:0wt} requiring that $\bK$ itself (rather than $\bK'$) act trivial-weight-freely are jointly equivalent \cite[Theorem B]{2506.09642v1} to $\bG$ containing a dense pointwise-elliptic sub\emph{set}, so we can assume that conjunction to hold throughout the present proof. 
  
  \begin{enumerate}[label={},wide]
    
  \item\textbf{\Cref{item:thn:dns.ell.sbgp:pt.ell.sbg} $\Rightarrow$ \Cref{item:thn:dns.ell.sbgp:pt.ell.sbs}} requires no argument. 
    
  \item\textbf{\Cref{item:thn:dns.ell.sbgp:pt.ell.sbs} $\Rightarrow$ \Cref{item:thn:dns.ell.sbgp:char}:} The unique \cite[\S 1.8]{ragh} maximal compact subgroup of the \emph{nil-radical} (\cite[Definition 10.40]{hm_pro-lie-bk}, \cite[Theorem 3.18.13]{var_lie}) $\bG_{nil}\le \bG$ is normal in $\bG$, so can be ignored by the proof's prefatory remarks. We can thus assume $\bG_{nil}$ simply-connected (and hence \cite[Theorem 3.6.2]{var_lie} analytically isomorphic to its Lie algebra via the exponential map). Thus:
    \begin{equation}\label{eq:gp.chn}
      \begin{tikzpicture}[>=stealth,auto,baseline=(current  bounding  box.center)]
        \path[anchor=base] 
        (0,0) node (g) {$\bG$}
        +(0,-1) node (gsol) {$\bG_{sol}$}
        +(0,-2) node (gnil) {$\bG_{nil}$}
        +(0,-3) node (1) {$\{1\}$}
        ;
        \draw[-] (g) to[bend left=0] node[pos=.5,auto] {$\scriptstyle $} (gsol);
        \draw[-] (gsol) to[bend left=0] node[pos=.5,auto] {$\scriptstyle $} (gnil);
        \draw[-] (gnil) to[bend left=0] node[pos=.5,auto] {$\scriptstyle $} (1);        
        \draw[decorate,decoration = {brace},thick] (1,0) to[bend left=0] node[pos=.5,auto] {$\scriptstyle\quad \bG/\bG_{sol}\cong \bK'/\left(\text{finite}\right)$} (1,-.8);
        \draw[decorate,decoration = {brace},thick] (1,-1) to[bend left=0] node[pos=.5,auto] {$\scriptstyle\quad \bG_{sol}/\bG_{nil}\cong \bR^d\times \bT^e$} (1,-1.8);
        \draw[decorate,decoration = {brace},thick] (1,-2) to[bend left=0] node[pos=.5,auto] {$\scriptstyle\quad \bG_{nil}\text{ nilpotent simply-connected}$} (1,-2.8);
        \draw[decorate,decoration = {brace},thick] (-1,-1.8) to[bend left=0] node[pos=.5,auto] {$\scriptstyle\quad \text{\parbox{3.5cm}{$\bG/\bG_{nil}$ contains (the image of) $\bK$ as a maximal compact subgroup}}$} (-1,0);
      \end{tikzpicture}
    \end{equation}
    where the abelianness of $\bG_{sol}/\bG_{nil}$ follows from \cite[Theorem 3.8.3(iii)]{var_lie}.

    
    As indicated in \Cref{eq:gp.chn}, we identify $\bK$ with its image through
    \begin{equation*}
      \bG
      \xrightarrowdbl{\quad\pi\quad}
      \overline{\bG}:=\bG/\bG_{nil}. 
    \end{equation*}
    Observe next that
    \begin{itemize}[wide]
    \item the commutator subgroup $\overline{\bG}^{((1))}\le \overline{\bG}$ is contained in $\bR^d\cdot \bK'$;

    \item and the elements of $\bK'$ are \emph{precisely} \cite[Theorem 9.2]{hm5} the commutators $[s,t]:=sts^{-1}t^{-1}$, $s,t\in \bK$.
    \end{itemize}
    $\Gamma$ being dense with pointwise-elliptic commutator set $[\Gamma,\Gamma]$, it follows that arbitrary elements
    \begin{equation*}
      (v,s)\in \bG_{nil}\cdot \bK'\cong \bG_{nil}\rtimes \bK'
    \end{equation*}
    are arbitrarily approximable by elliptic $t\in [\Gamma,\Gamma]$ with
    \begin{equation*}
      \pi(t)\in \bR^d\cdot \bK'\le \overline{\bG}=\bG/\bG_{nil}.
    \end{equation*}
    The maximal compact subgroups of $\bR^d\cdot \bK'$ are conjugate to $\bK'$, so we may assume all such $\pi(t)$ contained in $\bK'$ by \cite[Theorem 1.12]{2506.09642v1}. We can thus decompose
    \begin{equation}\label{eq:tvs}
      t=(v_t,s_t)\in \bG_{nil}\cdot \bK'\cong \bG_{nil}\rtimes \bK'
      ,\quad
      v_t\in \left\{w^{-1}\cdot \Ad_{s_t}w\ :\ w\in \bG_{nil}\right\}      
    \end{equation}
    (the latter inclusion claim following from the assumed ellipticity of $t$ and \cite[Lemma 1.7(1)]{2506.09642v1}).

    Choose $s\in \bK'$ so as to minimize $\dim\ker\left(1-\Ad_s|_{\fg_{nil}}\right)$. For sufficiently $s'$ ranging over a sufficiently small neighborhood $U\ni s\in \bK'$ the decompositions
    \begin{equation*}
      \fg_{nil}
      =
      \ker\left(1-\Ad_{s'}\right)
      \oplus
      \im\left(1-\Ad_{s'}\right)
    \end{equation*}
    glue into a decomposition $U\times \fg_{nil}\cong \cE_{\ker}\oplus \cE_{\im}$ as a sum of trivial bundles, which the analytic isomorphism $\fg_{nil}\cong \bG_{nil}$ induced by the exponential map then turns into a decomposition
    \begin{equation*}
      U\times \bG_{nil}
      \cong
      \cF_{\ker}\times_{U}\cF_{\im}
    \end{equation*}
    as a \emph{Whitney product} \cite[\S 4]{zbMATH03056184} of trivial fiber bundles, with respective fibers
    \begin{equation*}
      \begin{aligned}
        \cF_{\ker,s'}
        &=\left\{g\in \bG_{nil}\ :\ \Ad_{s'}g=g \right\}\\
        \cF_{\im,s'}
        &=\left\{h^{-1}\cdot \Ad_{s'}h\ :\ h\in \bG_{nil} \right\}.
      \end{aligned}      
    \end{equation*}
    It follows that if $\dim\ker\left(1-\Ad_s|_{\fg_{nil}}\right)$ is minimal for $s\in \bK'$, then the $(v_t,s_t)$ of \Cref{eq:tvs} cannot approach $(v,s)\in \bG_{nil}\cdot \bK'$, $\Ad_s v=v$ unless $v$ is trivial.

    In summary: $\bK'$ must act on $\fg_{nil}$ trivial-weight-freely. This suffices, given that the action on the $\bR^d\le \fg_{sol}/\fg_{nil}$ of \Cref{eq:gp.chn} has no trivial weights by \cite[Theorem 1.3]{2506.09642v1}. 
    
  \item\textbf{\Cref{item:thn:dns.ell.sbgp:char} $\Rightarrow$ \Cref{item:thn:dns.ell.sbgp:pt.ell.sbg}:} We retain the standing assumption that $\bG$ is Lie. A further bit of simplification will be to substitute for $\bG$ the middle term of the top row in
    \begin{equation*}
      \begin{tikzpicture}[>=stealth,auto,baseline=(current  bounding  box.center)]
        \path[anchor=base] 
        (0,0) node (l) {$\bG_{sol}$}
        (-1.5,0) node (ll) {$\{1\}$}
        +(2.8,.5) node (um) {$\bullet$}
        +(2.8,-.5) node (dm) {$\bG$}
        +(4.5,.5) node (ur) {$\bK'$}
        +(4.5,-.5) node (dr) {$\bG/\bG_{sol}$}
        +(6.5,0) node (rr) {$\{1\}$}
        ;
        \draw[->] (ll) to[bend left=0] node[pos=.5,auto] {$\scriptstyle $} (l);
        \draw[right hook->] (l) to[bend left=6] node[pos=.5,auto] {$\scriptstyle $} (um);
        \draw[right hook->] (l) to[bend right=6] node[pos=.5,auto] {$\scriptstyle $} (dm);
        \draw[->] (um) to[bend left=6] node[pos=.5,auto] {$\scriptstyle $} (ur);
        \draw[->] (dm) to[bend right=6] node[pos=.5,auto] {$\scriptstyle $} (dr);
        \draw[->] (ur) to[bend left=6] node[pos=.5,auto,swap] {$\scriptstyle $} (rr);
        \draw[->] (ur) to[bend left=6] node[pos=.5,auto,swap] {$\scriptstyle $} (rr);
        \draw[->] (dr) to[bend right=6] node[pos=.5,auto,swap] {$\scriptstyle $} (rr);
        \draw[->>] (um) to[bend left=0] node[pos=.5,auto,swap] {$\scriptstyle $} (dm);
        \draw[->>] (ur) to[bend left=0] node[pos=.5,auto,swap] {$\scriptstyle $} (dr);
      \end{tikzpicture}
    \end{equation*}
    thus ensuring the splitting $\bG\cong \bG_{sol}\rtimes \bK'$.

    We will work with finitely-generated dense subgroups $\Gamma\le \bG$, which do always exist: select, say, f.g. topologically-generating subgroups for each of the individual factors in the decomposition \cite[\S 4.13, Theorem, pp.188-189]{mz}
    \begin{equation*}
      \bG = \bL_1\cdot\bL_2\cdots \bL_k\cdot \bK
      ,\quad
      \begin{aligned}
        \bK&\le \bG\text{ maximal compact}\\
        \bR\cong \bL_i&=\exp \fl_i\text{ for lines }\fl_i\le \fg:=Lie(\bG)
      \end{aligned}
    \end{equation*}
    Now, having fixed a sufficiently large $d\in \bZ_{\ge 0}$:
    \begin{enumerate}[(a),wide]
    \item\label{item:thn:dns.ell.sbgp:pf.gen.resid} \cite[Lemmas 6.3 and 6.4]{zbMATH05117945} ensure that the (non-empty, as just seen) set
      \begin{equation*}
        \left\{
          (s_i)_{i=1}^d\in \bG^d
          \ :\
          \overline{\Braket{s_i,\ 1\le i\le d}}
          =
          \bG
        \right\}
        \subseteq
        \bG^d
      \end{equation*}
      of topologically-generating $d$-tuples is at least what might be termed \emph{cluster-residual}: each of its members has a neighborhood where the set is residual. It follows from this and the fact \cite[Proposition 8.2]{zbMATH05117945} that for \emph{compact} connected Lie groups free tuples generating dense subgroups is residual that this is so also for compact-by-solvable connected Lie groups (i.e. compact quotient, solvable kernel), which $\bG$ is by assumption. 
      

    \item\label{item:thn:dns.ell.sbgp:pf.fr.resid} On the other hand, the set of $d$-tuples generating free subgroups of non-solvable connected Lie groups is residual in any case by \cite[Theorem and its proof]{zbMATH03351914}; so, then, is the set of $d$-tuples whose $\bK'$-components in $\bG\cong \bG_{sol}\rtimes\bK'$ generate free groups therein.
    \end{enumerate}
    It follows that for non-solvable \emph{Lie} groups satisfying \Cref{item:thn:dns.ell.sbgp:char} the conclusion holds in the stronger form of \Cref{corn:when.lie}. Recall, in order to conclude, that the set of $d$-tuples in the semisimple compact Lie group $\bK'$ generating free groups whose non-trivial elements satisfy the 1-eigenvalue condition of \Cref{le:im.ndg} is residual: this follows from \cite[Theorem 2]{zbMATH03845869} and the accompanying \cite[Remark 5.(i), p.160]{zbMATH03845869}. Combining this with \Cref{item:thn:dns.ell.sbgp:pf.gen.resid} and \Cref{item:thn:dns.ell.sbgp:pf.fr.resid} yields the desired conclusion. 

    Observing that there is nothing to prove when $\bG$ is solvable (for in that case condition \Cref{item:thn:dns.ell.sbgp:0wt} requires compactness), we are done. The symbol `($\bullet$)' henceforth stands for the mutually-equivalent cluster consisting of the statement's three conditions. 

  \item\textbf{The two versions of \Cref{item:thn:dns.ell.sbgp:alm.semidir}} are mutually equivalent because the exponential function $\fl\xrightarrow{\exp}\bL$ is an analytic diffeomorphism for connected, simply-connected nilpotent Lie groups \cite[Theorem 3.6.2.]{var_lie}. 

  \item\textbf{($\bullet$) $\Leftrightarrow$ \Cref{item:thn:dns.ell.sbgp:alm.semidir}:} The opening paragraph of \cite{MR1044968} effectively argues that $\bG$ fits into an extension \Cref{eq:lk.by.m} as soon as its set of elliptic elements is dense. \Cref{item:thn:dns.ell.sbgp:char} and \Cref{item:thn:dns.ell.sbgp:alm.semidir} both being invariant under quotienting out compact normal subgroups, we may as well specialize to groups of the form $\bL\rtimes \bK$ for compact Lie $\bK$ and connected, simply-connected nilpotent $\bL$. That for such groups \Cref{item:thn:dns.ell.sbgp:char} and \Cref{item:thn:dns.ell.sbgp:alm.semidir} are equivalent, though, is immediate.  \qedhere
  \end{enumerate}
\end{thn:dns.ell.sbgp}


To put the following statement in perspective, recall \cite[Proposition 1.18]{2506.09642v1} that compact groups acting on connected locally compact groups always leave (some) maximal compact subgroups thereof invariant. 

\begin{lemma}\label{le:im.ndg}
  Let $\bK$ be a compact group acting on a solvable connected Lie group $\bL$, and $\bT\le \bL$ a $\bK$-invariant maximal torus. Every element of $\bL\rtimes \bK$ whose $\bK$-component has no 1-eigenvectors in
  \begin{equation*}
    \fl/\ft
    ,\quad
    \fl:=Lie(\bL)
    ,\quad
    \ft:=Lie(\bT)
  \end{equation*}
  is elliptic. 
\end{lemma}
\begin{proof}
  \cite[Lemma 1.7]{2506.09642v1} settles the abelian-$\bL$ case so we can assume, inductively on the length of $\left(\bL^{((n))}\right)_n$, that an element
  \begin{equation*}
    t
    =
    (v,s)\in \bG:=\bL\rtimes \bK
    ,\quad
    1-\Ad_s\text{ invertible on }\fl/\ft
  \end{equation*}
  has elliptic image in the quotient $\bL/\bL^{((\mathrm{max}))}$ by the bottom (abelian) layer of the derived series. Additionally, since quotienting out compact normal subgroups is harmless (as observed in the introductory remarks to the proof of \Cref{thn:dns.ell.sbgp}), we can assume there is no torus component to $\bL^{((\mathrm{max}))}\cong \bR^d$. 

  The image of $t$ in $\bG/\bL^{((\mathrm{max}))}$ belongs to some maximal compact subgroup thereof, which surjects onto $\bK$ and hence splits as
  \begin{equation*}
    \bM\rtimes \bK
    ,\quad
    \bM\le \bL/\bL^{((\mathrm{max}))}\text{ maximal compact}.
  \end{equation*}
  We can thus substitute for $\bG$ the (automatically split \cite[\S VII.3.2, Proposition 3]{bourb_int_7-9_en}) extension
  \begin{equation*}
    \{1\}
    \to
    \bL^{((\mathrm{max}))}
    \lhook\joinrel\xrightarrow{\quad}
    \bL^{((\mathrm{max}))}\rtimes \left(\bM\rtimes \bK\right)
    \xrightarrowdbl{\quad}
    \bM\rtimes \bK
    \to \{1\},
  \end{equation*}
  again falling back on the abelian-$\bL$ case to conclude.
\end{proof}

\addcontentsline{toc}{section}{References}

\def\polhk#1{\setbox0=\hbox{#1}{\ooalign{\hidewidth
  \lower1.5ex\hbox{`}\hidewidth\crcr\unhbox0}}}

\Addresses

\end{document}